\documentclass[11pt]{amsart}
\usepackage{amssymb,amsfonts,amsmath,amsthm}
\usepackage[all,2cell]{xy}
\usepackage{graphicx}
\usepackage{pstricks}
\usepackage[mathscr]{eucal}
\usepackage{amscd}
\usepackage{epsfig}
\usepackage{enumerate}
\usepackage{indentfirst}
\usepackage{fancyhdr}

\newtheorem{thm}{\bf Theorem}[section]
\newtheorem{prop}[thm]{\bf Proposition}
\newtheorem{lem}[thm]{\bf Lemma}
\newtheorem{cor}[thm]{\bf Corollary}

\theoremstyle{definition}
\newtheorem{dfn}[thm]{\bf Definition}
\newtheorem{ex}[thm]{\bf Example}

\theoremstyle{remark}
\newtheorem{rem}[thm]{\bf Remark}

\def \Q{\mathbb{Q}}
\def \C{\mathbb{C}}
\def \R{\mathbb{R}}
\def \CP{\mathbb{C}{\rm P}}
\def \P{\mathbb{P}}
\def \S{\mathbb{S}}

\def \E{{\rm Eu}}

\newcommand{\acknowledge}{\subsection*{Acknowledgments}}

\title[Rational smoothness, cellular decompositions, GKM theory]{Rational smoothness, cellular decompositions and GKM theory}
\author{Richard Gonzales}

\address{
Department of Mathematics\\
Galatasaray University \\
Ciragan Cad. No.36, Besiktas\\ 
34357 Istanbul \\
TURKEY}         
\email{rgonzalesv@gmail.com}

\subjclass[2010]{Primary 14F43, 14L30; Secondary 55N91, 14M15}
\keywords{rational smoothness, algebraic torus actions, GKM theory, equivariant cohomology, algebraic monoids, group embeddings.}

\begin{document}


\begin{abstract}
We introduce the notion of $\Q$-filtrable varieties: projective
varieties with a torus action and a finite number of fixed points, such
that the cells of the associated Bialynicki-Birula decomposition
are all rationally smooth. Our main results develop GKM theory 
in this setting. We also supply a method for building nice combinatorial 
bases on the equivariant 
cohomology of any $\Q$-filtrable GKM variety.
Applications to the theory of group embeddings are provided.
%
%
\end{abstract}
\maketitle

\section*{Introduction and Statement of the Main Results}
\addcontentsline{toc}{section}{Introduction}
Let $X$ be a smooth complex projective algebraic variety with a $\C^*$-action and
finitely many fixed points $x_1,\ldots,x_m$. The method
of Bialynicki-Birula (\cite{bb:torus}) gives rise to a decomposition of $X$ into 
locally closed subvarieties: 
$$W_i=\{x\in X\,|\, \lim_{t\to 0}tx=x_i\}.$$
Clearly, $X=\bigsqcup_i W_i$. The subvarieties $W_i$ are called
cells of the decomposition.  
Theorem 4.3 of \cite{bb:torus} asserts that all cells 
are isomorphic to affine spaces, that is, $W_i\simeq \C^{n_i}$ for all $i$.
From this, one concludes that
$X$ has no cohomology in odd degrees.
This method for breaking down a projective variety into pieces, 
also known as {\em BB-decomposition}, 
allows to compute important topological invariants, e.g. Betti numbers.  
It is worth emphasizing that many of the ideas of \cite{bb:torus} extend to the singular case. 
In fact, the BB-decomposition makes sense even if $X$ is singular,
though, this time, the cells need not be so well behaved. 
%

\smallskip

Goresky, Kottwitz and MacPherson in their seminal paper \cite{gkm:eqc},
developed a theory, nowadays called GKM theory, that makes it possible 
to describe the equivariant cohomology of certain {\em $T$-skeletal} varieties: 
projective algebraic varieties upon which an algebraic torus $T$ acts
with a finite number of fixed points and invariant curves. 
Cohomology, in this article, is considered with rational coefficients. 
Let $X$ be a $T$-skeletal variety and denote by $X^T$ the fixed point set. 
The main purpose of GKM theory is to identify 
the image of the functorial map 
\[
i^* : H^*_T(X) \to H^*_T(X^T),
\]
assuming $X$ has no cohomology in odd degrees ({\em equivariantly formal}).
GKM theory asserts that 
if $X$ is a {\em GKM variety}, i.e. $T$-skeletal and equivariantly formal, then 
the equivariant cohomology ring $H^*_T(X)$ 
can be identified with certain ring of piecewise polynomial functions $PP_T^*(X)$ (Theorem \ref{gkm.thm}). 

\smallskip

Mostly, GKM theory has been applied to smooth projective $T$-skeletal varieties, 
because they all have trivial cohomology in odd degrees (BB-decomposition).
Furthermore, the GKM data
issued from the fixed points and invariant curves 
has been explicitly obtained
for some interesting subclasses: 
flag varieties (\cite{c:schu}, \cite{bri:eqchow}),
toric varieties (\cite{bri:ech}, \cite{vv:hkth}, \cite{uma:kth}) 
and regular embeddings of 
reductive groups 
(\cite{bri:bru}, \cite{uma:kth}). Additionally, GKM theory has been applied
to Schubert varieties (\cite{c:schu}, \cite{bri:eqchow}, \cite{bri:eu}). 
The latter ones, even though singular,
are GKM varieties and 
their BB-cells (relative to an appropriate action of $\C^*$)
are exactly the Bruhat cells.

\medskip

Let now $X$ be a complex algebraic variety of dimension $n$, and $x\in X$.  
We say that $X$ is {\em rationally smooth at}
$x$, if there exists a neighborhood $U$ of $x$ (in the complex topology)
such that, for all $y\in U$, we have 
$$
\begin{array}{cccccc}
H^m(X,X-\{y\})=(0) \;\;{\rm if}\; m\neq 2n, &{\rm and} \;\; H^{2n}(X,X-\{y\})=\mathbb{Q}.\\
\end{array}
$$
If $X$ is rationally smooth at every $x\in X$, then $X$ is called {\em rationally smooth}. 
%
%
Such 
varieties satisfy Poincar\'e duality with rational coefficients \cite{mc:hm}. 
See \cite{bri:rat} for an up-to-date discussion of rationally smooth singularities on 
complex algebraic varieties with torus action. 

\smallskip

Let $G$ be a connected reductive group.
Recall that a normal irreducible projective variety $X$ is called an {\em embedding} of $G$, 
or a {\em group embedding}, 
if $X$ is a $G\times G$-variety containing an open orbit 
isomorphic to $G$. Let $M$ be a reductive monoid with zero and unit group $G$.
Then there exists a central one-parameter subgroup $\epsilon:\C^*\to G$,
with image $Z$ contained in the center of $G$, 
such that $\displaystyle \lim_{t\to 0}\epsilon(t)=0$. 
Moreover,
the quotient space $$\P_\epsilon(M):=(M\setminus\{0\})/Z$$
is a normal projective embedding of the quotient group $G/Z$.
Embeddings of the form $\P_\epsilon(M)$ are called {\em standard group embeddings}.
It is known that all normal projective embeddings of a 
connected reductive group are standard (\cite{ab:em}).
Using methods from the theory of algebraic monoids, Renner (\cite{re:ratsm}, \cite{re:hpolyirr})
investigated those standard embeddings that are rationally smooth. 

\medskip

The purpose of this article is to 
%
establish GKM theory 
in the setting of {\em $\Q$-filtrable varieties:} projective varieties with a torus action 
having finitely many fixed points, such that the cells of the associated
BB-decomposition are all rationally smooth, i.e. rational cells.
In general, $\Q$-filtrable varieties have singularities. 
As an application of our theory, 
we show that rationally smooth standard
embeddings are $\Q$-filtrable. Our results 
lay down the topological foundations for the study of 
rationally smooth 
standard embeddings 
via GKM theory.

\medskip

This article is organized as follows.
The first two sections briefly review GKM theory. 
In Section 3, we devote ourselves to the study of rational cells 
and state their main topological features (Thm. \ref{topratcell.thm}, Prop. 3.11, and Thm. \ref{cell.curves.thm}). In
Section 4 we introduce the notion  $\Q$-filtrable varieties.  
Our main result in this section is given below.

\medskip

\noindent {\bf Theorem \ref{eqforfiltration.thm}.}
{\it
Let $X$ be a normal projective $T$-variety. Suppose that $X$ is $\Q$-filtrable. Then 
\begin{enumerate}[(a)]
\item $X$ admits a filtration into $T$-stable closed subvarieties $X_i$, $i=0, \ldots, m$, such that
$$\emptyset=X_0\subset X_1\subset \ldots \subset X_{m-1}\subset X_m=X.$$
\item each cell $C_i=X_i \setminus X_{i-1}$ is a rational cell, for $i=1,\ldots, m$.
\item For each $i=1,\ldots, m$, the singular rational cohomology of $X_i$ vanishes in odd degrees. 
      In other words, each $X_i$ is equivariantly formal.  
\item If, in addition, the $T$-action on $X$ is $T$-skeletal, then each $X_i$ is a GKM-variety.
\end{enumerate}  
}

It is worth noting that $\Q$-filtrable spaces need not be rationally smooth.
For instance, Schubert varieties admit a decomposition into affine
cells but they are not always rationally smooth. As a way of example, 
the Schubert variety of codimension one in the Grassmannian
of $2$-planes in $\C^4$ does not satisfy Poincar\'e duality, and hence 
is not rationally smooth. References  
\cite{ar:eu} and \cite{bri:ech} supply some criteria 
for rational smoothness of Schubert varieties.   
What is remarkable 
about 
$\Q$-filtrable varieties
is that 
they are equivariantly formal. 

In Section 6, after recalling Arabia's notion of equivariant Euler classes (Section 5),
we construct free module generators on the equivariant cohomology
of any $\Q$-filtrable GKM variety. 
Our findings extend the earlier works of Arabia(\cite{ar:eu}) and Guillemin-Kogan (\cite{gk:morse}). 
The main result of Section 6 is the following.

\medskip

\noindent{\bf Theorem \ref{eulergenerators.thm}.}
{\it Let $X$ be a $\Q$-filtrable GKM-variety.
Let $x_1<x_2<\ldots<x_m$ be the order relation on $X^T$ 
compatible with the filtration of $X$ given in Theorem \ref{eqforfiltration.thm}. 
Then there exist unique classes $\theta_i\in H^*_T(X)$, $i=1,\ldots,m$, with the following properties:
\begin{enumerate}[(i)]
\item $I_i(\theta_i)=1$,

\item $I_j(\theta_i)=0$ for all $j\neq i$,

\item the restriction of $\theta_i$ to $x_j \in X^T$ is zero for all $j<i$, and

\item $\theta_i(x_i)=\E_T(i,C_i)$.
\end{enumerate}
Moreover, the $\theta_i$'s generate $H^*_T(X)$ freely as a module over $H^*_T(pt)$. 
}
\smallskip

Here $I_i:H^*_T(X)\to H^*_T(pt)$ is the $H^*_T$-linear map
obtained by integrating, along $X_i$, the pullback
of a class in $X$ to $X_i$, and 
$\E_T(i,C_i)$ stands for the equivariant Euler class.

\smallskip

Although the class of $\Q$-filtrable varieties
includes
smooth projective $T$-skeletal varieties 
and Schubert varieties, 
its crucial attribute 
is that it also includes
a large and interesting family of
singular group embeddings, namely, 
rationally smooth standard embeddings.  
Indeed, in the last section of this article, we show that 
the notion of $\Q$-filtrable variety is well suited to the study of group embeddings
and, in doing so, we provide our theory with its major set of fundamental examples. 
Our main result in this direction can be stated as follows.

\medskip

\noindent{\bf Theorem \ref{ratsmisqfiltrable.cor}.}
{\it Let $X=\mathbb{P}_\epsilon(M)$ be a standard group embedding. 
If $X$ is rationally smooth, then $X$ is $\Q$-filtrable and so it has no cohomology in odd degrees.}


\smallskip

\acknowledge Some of these results were part of the author's doctoral dissertation under the supervision
of Lex Renner. I would like to thank him and Michel Brion for their invaluable help.

\section{Equivariant Cohomology and Localization}



Throughout this article, we work with complex algebraic varieties. 
Cohomology is always considered with rational coefficients.


\subsection{The Borel construction}
Let $T=(\C^*)^r$ be an algebraic torus and let $X$ be a 
$T$-variety, that is, a complex algebraic variety 
with an algebraic action of $T$.
Let $ET\to BT$ be a universal principal bundle for $T$. 
The {\bf equivariant cohomology} of $X$ (with rational coefficients) is defined to be
$$H_{T}^*(X):=H^*(X_T),$$
where $X_T=(X\times ET)/T$ is the total space associated to the fibration
$$\xymatrix{X \ar@{^(->}[r] & X_T \ar[r]^{p_X} & BT}.$$
This construction was introduced by Borel \cite{bo:sem}.
Here, $BT$ is simply connected, the map $p_X$ is induced by the canonical projection 
$ET\times X\to ET$, and $T$ acts diagonally on $ET\times X$.
Notice that
$H^*_T(X)$ is, via $p_X^*$, an algebra over $H^*_T(pt)$.
To simplify notation, we sometimes write $H^*_T$ instead of $H^*_T(pt)$.
%
%

It can be shown that 
$H^*_T(X)$ is independent of the choice of universal $T$-bundle.  
See \cite{bo:sem} and \cite{qui1:spec} for more details.

\begin{ex}
Let $T=(\C^*)^r$ be an algebraic torus. In this case, $BT=(\CP^\infty)^r$, and consequently 
$H^*_T(pt)=H^*(BT)=\Q[x_1,\ldots,x_r]$,
where $deg(x_i)=2$. 
A more intrinsic description of $H^*_T(pt)$ is given as follows. 
Denote by $\Xi(T)$ the character group of $T$. 
Any $\chi \in \Xi(T)$ defines a one-dimensional complex representation
of $T$ with space $\C_\chi$. Here $T$ acts on $\C_\chi$ via $t\cdot z:=\chi(t)z$. 
Consider the associated
complex line bundle $$L(\chi):=(E_T\times_T \C_\chi\to BT)$$
and its first Chern class $c(\chi)\in H^2(BT)$. Let $S$
be the symmetric algebra over $\Q$ of the group $\Xi(T)$. Then
$S$ is a polynomial ring on $r$ generators of degree $1$, and the map
$\chi \to c(\chi)$ extends to a ring isomorphism $$c:S\to H^*_T(pt)$$
which doubles degrees: the 
{\em characteristic homomorphism} (\cite{bo:sem}).
\end{ex}


\subsection{Localization Theorem for torus actions}

%
%
Let $S\subset H^*_T$ be the multiplicative system $H^*_T\setminus \{0\}$. 
For a given $T$-variety $X$, denote by $X^T$ the fixed point set.
The following is a classical theorem due to 
Borel (\cite{bo:sem}). See also \cite{hs:ctg}, Theorem III.1. 

\begin{thm}
Let $X$ be a $T$-variety.
Suppose $H^*_T(X)$ is a finite $H^*_T$-module.
Then the localized restriction homomorphism
$$S^{-1}H^*_T(X)\longrightarrow S^{-1}H^*_T(X^T)=H^*(X^T)\otimes_{\Q}(S^{-1}H^*_T)$$
is an isomorphism. \hfill $\square$
\end{thm}


\section{GKM theory}

GKM theory is 
a relatively recent tool 
that owes its name to the work of Goresky, Kottwitz and MacPherson \cite{gkm:eqc}. 
This theory encompasses techniques that date back to the early works of 
Atiyah (\cite{a:comp}, \cite{as:index2}), Segal (\cite{se:equiv}), Borel (\cite{bo:sem}) and Chang-Skjelbred (\cite{cs:schur}).

\subsection{Equivariant formality} 

\begin{dfn}
Suppose an algebraic torus $T$ acts on a (possibly singular) space $X$. Let $p_X:X_T\longrightarrow BT$ be the fibration
associated to the Borel construction. We say that $X$ is {\bf equivariantly formal} if the Serre spectral sequence
$$E_2^{p,q}=H^p(BT;H^q(X))\Longrightarrow H^{p+q}_T(X)$$
for this fibration degenerates at $E_2$.
\end{dfn}

The following theorem characterizes equivariant formality. 
For a proof, see \cite{gkm:eqc}, Theorem 1.6.2, or \cite{bri:eu}, Lemma 1.2.

\begin{thm}\label{ch.formal.thm}
Consider the following conditions for a $T$-variety $X$. 
\smallskip

\begin{enumerate}[(a)]
\item $X$ is equivariantly formal.

\smallskip

\item The edge homomorphism $H_T^*(X)\longrightarrow H^*(X)$
is surjective; that is, the ordinary rational cohomology is given by extension of scalars, $$H^*(X)\simeq H^*_T(X)\otimes_{H^*_T}\Q.$$

\smallskip

\item $H^*_T(X,\mathbb{Q})$ is a free $H^*_T(pt)$-module.


\smallskip

\item The singular rational cohomology of $X$ vanishes in odd degrees. 
\end{enumerate}
\medskip

\noindent Then (a)$\Leftrightarrow$(b)$\Leftrightarrow$(c)$\Leftarrow$(d).

\smallskip

If $X^T$ is finite, then all these conditions are equivalent. \hfill $\square$
\end{thm}

It follows that $X$ is equivariantly formal if and only if 
there is an isomorphism of $H^*_T$-modules between 
$H^*_T(X)$ and $H^*(X)\otimes_\Q H^*_T$. Every smooth
projective $T$-variety is equivariantly formal (\cite{gkm:eqc}, Theorem 14.1 (7)).



\smallskip

A joint application of Theorem 1.2 and Theorem 2.2 leads to the following.

\begin{cor} 
Let $X$ be a $T$-variety with a finite number of fixed points. 
Then $X$ is equivariantly formal if and only if $H^*_T(X)$ is a free $H^*_T$-module of rank $|X^T|$,
the number of fixed points.  \hfill $\square$
\end{cor}


\subsection{$T$-Skeletal Actions}





\begin{dfn} \label{genericaction.def} 
Let $X$ be a projective $T$-variety. 
Let $\mu:T\times X\to X$ be the action map.
We say that $\mu$ is a {\boldmath $T$\bf-skeletal action} if 
%
\begin{enumerate}
 \item $X^T$ is finite, and
 \item The number of one-dimensional orbits of $T$ on $X$ is finite.
\end{enumerate}
In this context, $X$ is called a {\boldmath$T$\bf-skeletal variety}. 
If a $T$-skeletal variety $X$ is also equivariantly formal, 
then we say that $X$ is a {\bf GKM variety}.
\end{dfn}

Let $X$ be a normal projective $T$-skeletal variety. 
Then $X$ has an equivariant
embedding into a projective space with a linear action of $T$ (\cite{su:eq}, Theorem 1),
and so the closure of any orbit of dimension one in $X$ contains exactly two fixed points.
Accordingly, it is possible to define a ring $PP_T^*(X)$
of {\bf piecewise polynomial functions}. Indeed, let $R=\bigoplus_{x\in X^T}R_x$,
where $R_x$ is a copy of the polynomial algebra $H^*_T$. We then define
$PP_T^*(X)$ as the subalgebra of $R$ defined by
\[
PP_T^*(X)= \{(f_1,...,f_n)\in\bigoplus_{x\in X^T}R_x\;|\; f_i\equiv f_j\;mod(\chi_{i,j})\}
\]
where $x_i$ and $x_j$ are the two {\em distinct} fixed points in the closure of the one-dimensional
$T$-orbit $\mathcal{C}_{i,j}$, and $\chi_{i,j}$ is the character of $T$ associated with $\mathcal{C}_{i,j}$.
The character $\chi_{i,j}$ is uniquely determined up to sign (permuting the two
fixed points changes $\chi_{i,j}$ to its opposite). 


\begin{thm}[\cite{cs:schur}, \cite{a:comp}, \cite{gkm:eqc}]\label{gkm.thm}
Let $X$ be a normal projective $T$-skeletal variety. Suppose that $X$ is a GKM variety.
Then the restriction mapping 
$$H^*_T(X)\longrightarrow H^*_T(X^T)= \bigoplus_{x_i\in X^T}H^*_T$$ 
is injective, and its image is the subalgebra $PP_T^*(X)$. 
\hfill $\square$
\end{thm}

%
%

\smallskip

Theorem \ref{ch.formal.thm}
characterizes normal projective 
GKM-varieties among all $T$-skeletal varieties.

\begin{thm}
Let $X$ be a normal projective variety with a $T$-skeletal action $\mu:T\times X \to X.$ 
Then $X$ is a GKM-variety
if and only if $X$ has no (rational) cohomology in odd degrees. \hfill $\square$
\end{thm}

We will show that the class of equivariantly formal spaces incorporates  
certain subclass of singular varieties, namely, $\Q$-filtrable varieties 
(Theorems \ref{eqforfiltration.thm} and \ref{eulergenerators.thm}). 
This subclass encompasses all rationally smooth 
standard embeddings of a reductive group (Theorem \ref{ratsmisqfiltrable.cor}).
As such, it is much larger than the subclass of smooth varieties.

\section{Rational Cells}

This section is devoted to the study of 
our most important topological tool: rational cells.





\begin{dfn}
Let $X$ be an algebraic variety with an action of a torus $T$ and a fixed point $x$.
We say that $x$ is an {\bf attractive fixed point} if there exists a one-parameter subgroup $\lambda:\C^*\to T$
and a neighborhood $U$ of $x$,
such that $\displaystyle \lim_{t\to 0}\lambda(t)\cdot y=x$ for all points $y$ in $U$.  
\end{dfn}

There is an important characterization of attractive fixed points. 
A proof of the following result can be found in \cite{bri:rat}, Proposition A2.

\begin{prop}\label{tang.cell}
For a torus $T$ acting on a variety $X$ with a fixed point $x$, the following conditions are equivalent:

(i) The weights of $T$ in the Zariski tangent space $T_x(X)$ are contained in an open half space.

(ii) There exists a one-parameter subgroup $\lambda:\C^*\to T$ such that, for all $y$ in a neighborhood of $x$, we have $\displaystyle \lim_{t\to 0}\lambda(t)\,y=\,x$.

If (ii) holds, then the set 
$$
X_x:=\{y\in X \,|\, \lim_{t\to 0}\lambda(t)\, y=\,x \}
$$
is the unique affine $T$-invariant open neighborhood of $x$ in $X$. 
Moreover, $X_x$ admits a closed $T$-equivariant embedding into $T_xX$. \hfill $\square$
\end{prop}

\begin{lem}
Let $X$ be an irreducible affine variety with a $T$-action and an attractive fixed point $x_0 \in X$.
Then $X$ is rationally smooth at $x_0$ if and only if $X$ is rationally smooth everywhere. 
\end{lem}

\begin{proof}
If $X$ is rationally smooth everywhere, 
then it is rationally smooth at $x_0$. 
For the converse, we use 
Proposition \ref{tang.cell} (ii) and the 
affineness of $X$ 
to guarantee the existence of 
a one-parameter subgroup $\lambda:\C^*\to T$
such that
$$X=\{y\in X \,|\, \lim_{t\to 0}\lambda(t)\, y=\,x_0 \}.$$
In symbols, $x_0\in \overline{\C^*\cdot y}$, for any $y\in X$.
Now consider the complex topology on $X$. We claim that 
any non-empty open  
$T$-stable subset of $X$ containing $x_0$ is all of $X$.
In effect, let $U$ be a $T$-stable neighborhood of $x_0$.
Then, for any $y\in X$,
there exists $s_y\in \C^*$, such that $s_y\cdot y \in U$.
Indeed, 
because $x_0$ is attractive, one can find
a sequence $\{t_n\}\subset \C^*$ such that $t_n\cdot y$ converges to $x_0$.
That is, there exists $N$ with the property that 
$t_N\cdot y$ belongs to $U$. Setting $s_y=t_N$ yields $s_y\cdot y\in U$. 
However, $U$ is $T$-stable, and therefore it contains the entire orbit 
$\C^*\cdot y$.
In short, $y\in U$ or, equivalently, 
$U=X$.

Hence, the non-empty open $T$-stable subset of rationally smooth points of $X$ is,
{\it a fortiori}, equal to $X$.
\end{proof}

\begin{dfn} \label{rational.cell}
Let $X$ be an irreducible affine variety with a $T$-action and an attractive fixed point $x_0\in X$. 
If $X$ is rationally smooth at $x_0$ (and thus everywhere),  
we refer to $(X,x_0)$ as a {\bf rational cell}.
\end{dfn}

It follows from Definition \ref{rational.cell} and Proposition \ref{tang.cell} 
that if $(X,x_0)$ is a rational cell, then 
$$X=\{y\in X \,|\, \lim_{t\to 0}\lambda(t)\, y=\,x_0 \},$$
for a suitable one-parameter subgroup $\lambda$. 
Notably, $\{x_0\}$ is the unique closed $T$-orbit in $X$.

\begin{ex}
Certainly $\C^n$ is a rational cell with the usual $\C^*$-action by scalar multiplication. 
Here the origin is the unique attractive fixed point.  
\end{ex}

\begin{ex}
Let $V=\{xy=z^2\} \subset \C^3$. 
The standard $\C^*$-action by scalar multiplication makes $V$ a rational cell 
with $(0,0,0)$ as its attractive fixed point. 
This is clear once we
observe that $V$ is the quotient of $\C^2$
by the finite group with two elements, where the non-trivial element
acts on $(s,t)\in \C^2$ via $(s,t)\mapsto (-s,-t)$. 
So Proposition A1 (iii) of \cite{bri:rat} implies that $V$ is rationally smooth.
\end{ex}

\begin{ex}
A normal variety is not necessarilly rationally smooth. For instance,
consider the hypersurface $H\subset \C^4$ defined by $\{xy=uv\}$. 
Because
the singular locus of $H$, namely $\{(0,0,0,0)\}$, has codimension three,
it follows that $H$ is normal (\cite{sha:alg}, p. 128, comments after Theorem II.5.1.3). 
Nevertheless,
$H$ is not rationally smooth at the origin.
To see this, let $T=(\C^*)^2$ act on $H$ via $(t,s)\cdot(x,y,u,v)=(tx,ts^2y,su,st^2v)$.
Then $H$ has the origin as its unique attractive fixed point.
Moreover, $H$ contains four $T$-invariant curves (the four coordinate axes) passing through $(0,0,0,0)$. 
If $H$ were rationally smooth at the origin, then, by a result of Brion (Theorem \ref{cell.curves.thm}),
the dimension of $H$ would equal the number of its $T$-invariant curves. This is a contradiction, since
$H$ is only three dimensional.
\end{ex}

\begin{dfn}
Let $Z$ be a rationally smooth complex projective variety. Let $n$ be the (complex) dimension of $Z$.
We say that $Z$ is a {\bf rational cohomology complex projective space} if
there is a ring isomorphism
$$H^*(Z)\simeq \Q[t]/(t^{n+1}),$$
where $deg(t)=2$. 
\end{dfn}

\medskip

Let $(X,x)$ be a rational cell. 
Then, by Proposition \ref{tang.cell}, $X$ admits a closed $T$-equivariant embedding into $T_xX$. 
Set $\dot{X}$ to be $X-\{x\}$. 
Choose an injective one-parameter subgroup $\lambda:\C^*\to T$ as in Definition \ref{rational.cell}. Then all weights of the $\C^*$-action
on $T_xX$ via $\lambda$ are positive. Thus, the quotient
$$\P(X):=\dot{X}/\C^*$$
exists and is a projective variety (\cite{bri:rat}). 
Indeed, it is a closed subvariety of $\mathbb{P}(T_xX)$, a weighted projective space. 
The variety $\P(X)$ can be viewed
as an algebraic version of the link of $X$ at $x$.

\smallskip

The following result, except for parts (b) and (c), is due to Brion (\cite{bri:rat}). 
The idea of the proof of part (b) is due to Renner.

\begin{thm}\label{topratcell.thm}
Let $(X,x_0)$ be a rational cell of dimension $n$. Then,

\medskip

\noindent (a) $X$ is contractible.

\medskip

\noindent (b) $X-\{x_0\}$ is homeomorphic to $\mathbb{S}(X)\times \R^+$, where $\S(X):=X-\{x_0\}/\R^+$ is a compact topological space.

\medskip

\noindent (c) $X-\{x_0\}$ deformation retracts to $\S(X)$. 
In addition, $X$ is rationally smooth at $x_0$ if and only if $X-\{x_0\}$, and thus $\S(X)$, is a rational cohomology sphere $\mathbb{S}^{2n-1}$.

\medskip

\noindent (d) The space $\P(X)=X-\{x_0\}/\C^*$ 
is a rationally smooth complex projective variety of dimension $n-1$. 
Furthermore, $X$ is rationally smooth if and only if $\P(X)$ is a rational cohomology complex projective space $\mathbb{CP}^{n-1}$.
%
\end{thm}

\begin{proof}
For part (a) simply notice that the action of 
$\C^*$ on $X$ extends to a map $\C\times X \to X$ sending $0\times X$ to $x_0$ and restricting to the identity $1\times X \to X$. 
Since the proof of (d) can be found in \cite{bri:rat}, Lemma 1.3,  it suffices to prove parts (b) and (c).

\smallskip

(b) 
%
%
%
From Proposition \ref{tang.cell}, we know that $X$ admits a closed $T$-equivariant embedding into $T_{x_0}X\simeq \C^d$, 
which identifies $x_0$ with $0$.
Choosing a one-parameter subgroup $\lambda:\C^*\to T$ as in Definition \ref{rational.cell} yields
a $\C^*$-action on $\C^d$ with only positive weights $m_1,\ldots, m_d$. Specifically, $\lambda\in \C^*$
acts on $\C^d$ via 
$$\lambda\cdot(z_1,\ldots, z_d)=(\lambda^{m_1}z_1,\ldots, \lambda^{m_d}z_d).$$

Next, define an $\R^+$-equivariant 
map
$N:\C^d\to \R$ 
by 
$$N(z_1,\ldots,z_d)=\sqrt{\sum_{i=1}^{d}(z_i\overline{z_i})^{1/{m_i}}}.$$ 
Clearly, for $\lambda \in \C$ and $z\in \C^d$, the definition favors $N(\lambda \cdot z)=|\lambda|N(z)$ 
(here $\lambda \cdot z$ means $(\lambda^{m_1}z_1,\ldots, \lambda^{m_d}z_d)$).
      

Since $\R^+$ acts freely on $X-\{0\}\subseteq \C^d-\{0\}$, the quotient map
$$X-\{0\}\to \S(X)$$
is a principal $\R^+$-fibration. 
We claim that this fibration is trivial, i.e. $$X-\{0\}\simeq \S(X)\times \R^+.$$ 

To prove the claim, we just need to provide a global section $s$.
In fact, we can do so canonically. Let $s:\S(X)\to X-\{0\}$ be the map defined by $$s([x])=\frac{1}{N(x)}\cdot x.$$
This map is well defined (given that we are using the $\C^*$-action mentioned above) and not only defines a global section,
but also a homeomorphism between $\S(X)$ and $X\cap N^{-1}(1)$, where $N^{-1}(1)$ is the ``unit" sphere. 
Thus, $\S(X)$ is compact.

\medskip

(c) The first claim follows immediately from part (b). 
As for the second assertion, remember that $X$ is contractible. Thus,
the long exact sequence of the pair $(X,X-\{x_0\})$
splits into short exact sequences
$$
0\longrightarrow H^*(X-\{x_0\}) \longrightarrow H^{*+1}(X,X-\{x_0\})\longrightarrow 0.
$$
Therefore $X$ is rationally smooth if and only if $X-\{x_0\}$ is a rational homology sphere of dimension $2n-1$.
\end{proof}


\begin{cor}
Keeping the same notation as in Theorem \ref{topratcell.thm}, 
the rational cell $X$ is homeomorphic to the open cone over $\S(X)$. Moreover, $\P(X)$ is equivariantly formal.
\end{cor}

\begin{proof}
The first assertion follows 
from Theorem \ref{topratcell.thm} (c). 
As for the second, 
by Theorem \ref{topratcell.thm} again, 
$\P(X)$ is a rational cohomology complex projective space and thus
has no cohomology in odd degrees. 
Theorem \ref{ch.formal.thm} concludes the proof. 
\end{proof}

\begin{prop}
Let $(X,x_0)$ be a rational cell of dimension $n$. Denote by $X^+$ its one point compactification. Then
$X^+$ is simply connected and has the rational homotopy type of $\mathbb{S}^{2n}$, the Euclidean $2n$-sphere. 
\end{prop}

\begin{proof}
First, observe that $X^+$ is path-connected.
As a consequence of Theorem \ref{topratcell.thm}, we can write $X^+$ as a union of two open cones $D_0$ and $D_\infty$; 
namely, $D_0=S\times [0,1)/S\times \{0\}$ and $D_\infty=S\times (\epsilon, \infty]/S\times \{\infty\}$,
where $S$ stands for $\S(X)=(X\setminus \{x_0\})/\R_+$, and $\epsilon$ is a positive number less than $1$. 
Given that $X-\{x_0\}$ is path-connected, the intersection $D_0\cap D_\infty = S\times (\epsilon, 1)$ is path-connected as well. 
In summary, $X^+$ can be written as the union of two contractible open subsets
with path-connected intersection. 
Thus, by van Kampen's theorem, $X^+$ itself is simply connected.
To finish the proof, 
we need to show that $X^+$ is a rational cohomoloy $2n$-sphere. 
This is a simple exercise, using the Mayer-Vietoris sequence of 
the cover $\{D_0,D_\infty\}$.
\end{proof}

\begin{lem}[One-dimensional rational cells]\label{1dim.ch}
Let $(X,x)$ be a rational cell of dimension one. Then
\begin{enumerate}
\item $X$ is a cone over a topological circle.

\item $X$ is homeomorphic to $\C$.

\item If, additionally, $X$ is normal, then $X$ is isomorphic to $\C$ as an algebraic variety.
\end{enumerate}
\end{lem}

\begin{proof}
Without loss of generality, we can assume that $T$ acts faithfully on $X$.
Thus, $T$ is isomorphic to $\C^*$.
Now assertions (1) and (2) can be proved as follows. 
Since $X$ is one-dimensional, then the singular locus 
is an invariant discrete set.
Nonetheless, $x_0$
is the unique attractive fixed point, and $\C^*$ is {\em connected}, 
so
the singular locus is either empty or consists of only one point, namely, $x_0$.
As a result, $X\setminus \{x_0\}$
is smooth. 
Next notice that $X$ has two $\C^*$-orbits: 
the attractive fixed point $x_0$, and a dense open orbit of the form $\C^*$.
Hence, $X$ is homeomorphic to $\C$ 
and it is a cone over the circle $S^1$.

Finally, if we also assume that $X$ is normal and one-dimensional, then {\em a fortiori} $X$ is smooth (\cite{har:ag}). 
This proves (3).
\end{proof}

\begin{lem}
Let $(X,x)$ be a rational cell. Suppose $x$ is a smooth point. Then $X$ is isomorphic to its tangent space at $x$.
\end{lem}
\begin{proof}
By Proposition \ref{tang.cell}, we know that $X$ admits an equivariant {\em closed} embedding into $T_x X$. 
If $x$ is a smooth point, 
then both $X$ and $T_x X$ have the same dimension. 
For affine varieties this can only happen if $X=T_x X$.
\end{proof}


We are now ready to state what we call the
Equivariant 
Normalization Theorem for rational cells.
It is due to Brion (\cite{bri:ech}, Proof of Theorem 18, implication $(i)\Rightarrow(ii)$) 
and Arabia (\cite{ar:eu}, Section 3.2.1).

\begin{thm} \label{arabiamap.thm}
Let $(X,x)$ be a rational cell. 
Then there exists a $T$-module $V$ 
and an equivariant finite surjective map $\pi:X\to V$ such that
$\pi(x)=0$ and $V^T=\{0\}$.  \hfill $\square$
\end{thm}


We now specialize a result of Brion (\cite{bri:rat}, Section 1.4, Corollary 2) 
to rational cells.

\begin{thm}\label{cell.curves.thm}
Let $(X,x)$ be a rational cell. 
Suppose that the number of closed irreducible $T$-stable curves
on $X$ is finite. Let $n(X,x)$ be this number. 
Then $n(X,x)=dim(X).$ \hfill $\square$
\end{thm}



\section{Homology and Betti numbers of $\Q$-filtrable spaces}

We define $\Q$-filtrable varieties,
spaces that come equipped with a paving
by rational cells, and show that
they are equivariantly formal (Theorem \ref{eqforfiltration.thm}).

\subsection{The Bialynicki-Birula decomposition}
Let $X$ be a projective algebraic variety with a $\mathbb{C}^*$-action and
a finite number of fixed points $x_1,\ldots,x_m$. Consider the associated
BB-decomposition $X=\bigsqcup_i W_i$, where each cell is defined as follows
$$W_i=\{x \in X\,|\, \displaystyle\lim_{t\to 0}{t\cdot x=x_i}\}.$$

\smallskip

In the present section, we show that rational cells are 
a good substitute for 
the notion of affine space in the topological study of singular varieties. 


\begin{rem}
In general, the $BB$-decomposition of a projective variety is not a stratification; that is, it may happen that the closure of a cell
is not the union of cells, even if we assume our variety to be smooth. For a justification of this claim, see \cite{bb:decomp}.
\end{rem}


\begin{dfn}
Let $X$ be a complex algebraic variety endowed with a $\mathbb{C}^*$-action
and a finite number of fixed points.
A {\bf BB-decomposition} $\{W_i\}$ is said to be {\bf filtrable} 
if there exists a finite increasing sequence $X_0\subset X_1\subset \ldots \subset X_m$
of closed subvarieties of $X$ such that:

\medskip

\noindent a) $X_0=\emptyset$, $X_m=X$,

\medskip

\noindent b) For each $j=1,\ldots, m$, the ``stratum'' $X_{j}\setminus X_{j-1}$ is a cell of the decomposition $\{W_i\}$.
\end{dfn}  

The following result is due to Bialynicki-Birula (\cite{bb:decomp}). 

\begin{thm}
Let $X$ be a normal projective variety with a torus action and a finite number of fixed points. 
Then the stable decomposition is filtrable. \hfill $\square$
\end{thm}

\smallskip

\subsection{$\Q$-filtrable spaces} To begin with, let us introduce a few technical results.

\begin{lem}
Let $X$ be a complex projective algebraic variety with a $\C^*$-action.
Suppose $X$ can be decomposed as the disjoint union
$$X=Y\sqcup C,$$
where $Y$ is a closed stable subvariety and $C$ is an open rational cell
containing a fixed point of $X$, say $c_0$, as its unique attractive fixed point.
Denote by $n$ the (complex) dimension of $C$. 
Then
$$H^{k}(X,Y)=
\left\{
\begin{array}{rcl}
0                       & {\rm if} &k\neq 2n \\
\Q                      & {\rm if} &k=2n. \\
\end{array}
\right.
$$
Furthermore, if $Y$ has vanishing odd cohomology, then 
$$H^k(X,\Q)=
\left\{
\begin{array}{rcl}
H^k(Y,\Q)                        & {\rm if} &k\neq 2n \\
H^{2n}(Y,\Q)\oplus \Q           & {\rm if} &k=2n. \\
\end{array}
\right.
$$
\end{lem}

\begin{proof}
Let $H^*_{c}(-)$ denote cohomology with compact supports.
It is well-known that $H^*(X)=H^*_{c}(X)$ and $H^*(Y)=H^*_{c}(Y)$,
because $X$ and $Y$ are complex projective varieties.
Moreover, by Corollary B.14 of \cite{ps:mixed}, 
one has $$H^*(X,Y)\simeq H^*_{c}(X-Y)=H^*_{c}(C).$$
Given that $C$ is a rational cell, and a cone over a rational cohomology sphere of dimension $2n-1$ (Corollary 3.10),
it follows easily that 
$$H^*_{c}(C)=H^*(C,C-\{c_0\})=\left\{
\begin{array}{rcl}
0                       & {\rm if} &k\neq 2n \\
\Q                      & {\rm if} &k=2n. \\
\end{array}
\right.
$$
So the first claim is proved.  

As for the second assertion, 
consider the long exact sequence of the pair $(X,Y)$, namely, 
$$
\ldots \rightarrow H^{*-1}(Y) \rightarrow H^{*}(X,Y)\rightarrow H^{*}(X) \rightarrow H^{*}(Y) \rightarrow H^{*+1}(X,Y) \rightarrow \ldots.
$$ 
By our previous remarks, this long exact sequence can be rewritten as
$$
\ldots \rightarrow H^{*-1}(Y) \rightarrow H^{*}_{c}(C)\rightarrow H^{*}(X) \rightarrow H^{*}(Y) \rightarrow H^{*+1}_{c}(C) \rightarrow \ldots.
$$ 
If $Y$ has no cohomology in odd degrees, then the long exact sequence splits, yielding 
the identifications $H^i(X)=H^i(Y)$, when $i\neq 2n$, and
$$H^{2n}(X)=H^{2n}(Y)\oplus H^{2n}_c(C)=H^{2n}(Y)\oplus \Q.$$
The proof is now complete.
\end{proof}

\begin{cor} Keeping the notation of Lemma 4.4, 
attaching a complex $n$-dimensional rational cell 
produces no changes in cohomology, except in degrees $2n-1$ and\, $2n$. 
Furthermore, if $Y$ has no cohomology in odd degrees, then $X$ has no odd cohomology either,
and there is a short exact sequence of the form
$$
0\longrightarrow H^{2n}_c(C)\longrightarrow H^{2n}(X)\longrightarrow H^{2n}(Y) \to 0.
$$
\end{cor}

\begin{proof}
We simply observe that the long exact sequence of
the pair $(X,Y)$ gives
$$H^{k}(X)\simeq H^k(Y)$$
for $k\neq 2n-1\,,\, 2n$. Besides, we also obtain the exact sequence
$$
0 \rightarrow H^{2n-1}(X)\rightarrow H^{2n-1}(Y) \rightarrow H^{2n}_c(C)=\Q\rightarrow H^{2n}(X) \rightarrow H^{2n}(Y)\rightarrow 0.
$$ 
So in general $H^{2n-1}(X)$ injects into $H^{2n-1}(Y)$. 
In case we assume $Y$ to have vanishing odd cohomology,
we obtain
$X$ with vanishing odd cohomology as well, and a ``lifting of generators'' sequence:
$$
0\longrightarrow H^{2n}_c(C)\longrightarrow H^{2n}(X)\longrightarrow H^{2n}(Y) \to 0.
$$
\end{proof}

\begin{cor}\label{qfilthom.cor}
Let $X$ be a {\it normal} complex projective variety endowed with a $\C^*$-action and a finite number of fixed points. 
Suppose that each BB-cell is a rational cell. 
%
Then $X$ has vanishing odd cohomology over the rationals,
and the dimension of its cohomology group in degree $2k$ equals the number of rational cells of complex dimension $k$.
Furthermore, $X$ is equivariantly formal and $\chi(X)=|X^T|$.
\end{cor}

\begin{proof}
Since the BB-decomposition on $X$ is filtrable, the result follows from the previous lemma
as we move up in the filtration by attaching one rational cell at the time. This process is systematic and preserves
cohomology in lower and higher degrees at each step.
\end{proof}

Let $T$ be an algebraic torus acting on a variety $X$. 
A one-parameter subgroup $\lambda:\C^*\to T$ is called {\em generic}
if $X^{\C^*}=X^T$, where $\C^*$ acts on $X$ via $\lambda$.
Generic one-parameter subgroups always exist. Note that the $BB$-cells
of $X$, obtained using $\lambda$, are $T$-invariant.

\medskip

Our results in this section suggest the following definition.


\begin{dfn}\label{qfiltrable.def}
Let $X$ be a projective variety equipped with a $T$-action. 
We say that $X$ is {\bf $\Q$-filtrable} if
\begin{enumerate}
 \item X is normal,
 \item the fixed point set $X^T$ is finite, and
 \item there exists a generic one-parameter subgroup $\lambda:\C^*\to T$ for which the associated $BB$-decomposition
of $X$ consists of 
{\em rational cells}. 
\end{enumerate}  
\end{dfn}

\begin{thm}\label{eqforfiltration.thm}
Let $X$ be a normal projective $T$-variety. Suppose that $X$ is $\Q$-filtrable. Then 
\begin{enumerate}[(a)]
\item $X$ admits a filtration into closed subvarieties $X_i$, $i=0, \ldots, m$, such that
$$\emptyset=X_0\subset X_1\subset \ldots \subset X_{m-1}\subset X_m=X.$$
\item each cell $C_i=X_i \setminus X_{i-1}$ is a rational cell, for $i=1,\ldots, m$.
\item For each $i=1,\ldots, m$, the singular rational cohomology of $X_i$ vanishes in odd degrees. 
      In other words, each $X_i$ is equivariantly formal.  
\item If, in addition, the $T$-action on $X$ is $T$-skeletal, then each $X_i$ is a GKM-variety.
\end{enumerate}  
\end{thm}

\begin{proof}
Assertions (a) and (b) are a direct consequence of Definition \ref{qfiltrable.def} and Theorem 4.3. 
Applying Corollary \ref{qfilthom.cor} and Theorem \ref{ch.formal.thm} at each step of the filtration 
yields claim (c). For statement (d), we argue as follows.
Notice that all the $X_i$'s have vanishing odd cohomology, as it is guaranteed by (c).
Moreover, since the $X_i$'s are $T$-invariant and the $T$-action on $X$ is $T$-skeletal, 
then 
each $X_i$ contains only a finite number of fixed points and $T$-invariant curves.
In consequence, Theorem 2.5 applied to each $X_i$ gives (d). 
\end{proof}

Thus, we obtain the applicability of GKM-theory at each step of
the filtration, even though the various $X_i$'s are not necessarily
rationally smooth. This approach is more flexible than the general 
approach (by comparing singular cohomology with intersection
cohomology), used, for instance, in Theorem 3.7 of \cite{re:hpolyirr} 
or in \cite{w:formal}. Such flexibility will become apparent
from our results in Section 6, where we supply a method 
for constructing free
module generators on the equivariant cohomology of  
$\Q$-filtrable GKM-varieties (Theorem \ref{eulergenerators.thm}). 
This method is based on the notion of equivariant Euler classes.



\section{Equivariant Euler classes}
We quickly review the theory of equivariant Euler classes.
For a complete treatment of the subject, the reader is invited to consult 
\cite{hs:ctg} and \cite{ar:eu}.

\medskip

Let $X$ be a $T$-variety with an isolated fixed point $x$. 
To fix ideas, let us first assume 
that $(X,x)$ is a rational cell of dimension $n$.
Recall that $\mathbb{S}(X)=[X-\{x\}]/\R^+$ is a rational cohomology sphere $\S^{2n-1}$
and that $X$ is homeomorphic to the (open)
cone over $\mathbb{S}(X)$ (Theorem \ref{topratcell.thm} and Corollary 3.10).
The Borel construction yields the fibration
$$
\mathbb{S}(X)\hookrightarrow \mathbb{S}(X)_T\longrightarrow BT.
$$
Observe that the $E_2$-term of the corresponding Serre spectral sequence 
consists of only two lines, namely,
$$E_2^{p,q}=H^p(BT)\otimes H^q(\mathbb{S}(X))\neq 0 {\rm \;  \; only \; \, when \; \,} q=0 {\rm \; \, and \; \,} q=2n-1.$$
Let $\E_T(x,X)\in H^{2n}(BT)$ be the transgression of the generator $\lambda_X$ of $H^{2n-1}(\mathbb{S}(X))$. 
We call $\E_T(x,X)$ the {\bf Equivariant Euler class of $X$ at $x$}.
It follows from \cite{hs:ctg}, Theorem IV.6, that $\E_T(x,X)$ splits into a product
of linear polynomials, namely
$$
\E_T(x,X)=\omega_1^{k_1}\cdots \omega_s^{k_s},
$$
where $\omega_i\in H^2(BT)\simeq \Xi(T)\otimes \Q$. Here $\Xi(T)$ stands for the character group of $T$, and the isomorphism
is provided in Example 1.1. 

\smallskip

Since $X$ is a cone over $\mathbb{S}(X)$, then
$H^*_{c}(X)\simeq H^*(X,X-\{x\})\simeq \Q$, where $H^*_c(-)$ denotes cohomology with compact supports.
Using the Serre spectral sequence, one notices that
these isomorphisms are also valid in equivariant cohomology:
$$
H^*_{T,c}(X)\simeq H^*_T(X,X-\{x\})\simeq H^*_T.
$$

Let $\mathcal{T}_X$ be the 
canonical generator of $H^*_T(X,X-\{x\})$.
This generator can be described by the 
commutative diagram
$$
\xymatrix{
H^*_T(X,X-\{x\})   \ar[rrr]^{i^*}  \ar@<.5cm>[d]^{\int_{[X]}} & &  & H^*_T(X) \ar[d]^{res} \\
H^*_T(x) \ar@<.5cm>[u]^{\Phi^*_{X}} \ar@{-->}[rrr]^{\times (\E_T(x,X))} & & & H^*_T(x), \\
}
$$
where $\Phi^*_X$ is multiplication by $\mathcal{T}_X$. In other words,
$\mathcal{T}_X$ is the unique class in $H^*_T(X,X-\{x\})$
whose restriction to $H^*_T(pt)$ coincides with $\E_T(x,X)$. 
It is customary in the literature to call $\mathcal{T}_X$ the {\bf Thom class of $X$}.
Let us bear in mind that the map $\Phi^*_X$ raises degree by $2n$. 
Clearly,
$H^*_{T}(X,X-\{x\})\simeq H^*_c(X)\otimes H^*_T(pt)$ and so, $H^j_{T,c}(X)=0$ for $j<2n$.
As for the integral appearing here, it is, by definition, the inverse of $\Phi^*_X$.

\smallskip

Let $\mathcal{Q}_T$ be the quotient field of $H^*_T$. 
If $\mu\in H^*_{T,c}(X)$, then
$$
\E_T(x,X)\wedge \int_{[X]}\mu=\mu_{x},
$$
where $\mu_{x}$ denotes restriction of the class $\mu$ to $x$.
Hence, the identity
$$
\frac{1}{\E_T(x,X)}=\frac{1}{\mu_{x}}\int_{[X]}\mu,
$$
holds in $\mathcal{Q}_T$, for every non-zero $\mu$ in $H^*_T(X,X-\{x\})$.

\medskip

More generally,
let $X$ be a $T$-variety with an isolated fixed point $x$.
Suppose that $X$ is rationally smooth at $x$.
As pointed out in \cite{ar:eu}, 
we can replace $X$ by a conical neighborhood $U_x$ of $x$ 
and define
$\E_T(x,X):=\E_T(x,U_x).$
For instance,
if $x$ is also an attractive fixed point, we can let 
$U_x$ be a rational cell (Proposition \ref{tang.cell}).

In case the isolated fixed point $x\in X$ is not necessarily a rationally smooth point,
Arabia (\cite{ar:eu})
has shown that we can still define an Euler class $\E_T(x,X)$.
The key ingredient here is that, by Theorem 1.2,
the map
$$i^*:H^*_T(X,X-\{x\})\to H^*_T(x)$$
is an isomorphism modulo $H^*_T$-torsion. Therefore,
the function that assigns to a non-torsion element 
$\mu\in H^*_T(X,X-\{x\})$ the fraction $\frac{1}{\mu_x}\int_X \mu \in \mathcal{Q}_T$
is {\em constant}. 

\begin{dfn}
Let $X$ be a $T$-variety. Suppose that $x\in X^T$ is an isolated fixed point. 
The fraction $$\frac{1}{\E_T(x,X)}:=\frac{1}{\mu_x}\int_{X}\mu \in \mathcal{Q}_T,$$
where $\mu$ is any non-torsion element of $H^*_T(X,X-\{x\})$,
is called the {\em inverse of the equivariant Euler class of $X$ at $x$}.
When this fraction is non-zero, we denote its inverse by $\E_T(x,X)$ and call it
the {\em Equivariant Euler class of $X$ at $x$.} 
\end{dfn}

From Remark 2.3-2(b) of \cite{ar:eu}, it follows that if 
$x$ is a rationally smooth point of $X$, then $\E_T(x,X)$ is a non-zero polynomial, 
and splits into a product of linear factors.

\begin{ex}
When $X=\C^n$, $x=0$, and the algebraic torus $T$ acts linearly on $\C^n$, one proves
$$\E_{T}(0,\C^n)=(-1)^n\prod_{\alpha \in \mathcal{A}} \alpha,$$
where $\mathcal{A}$ is the collection of weights. Furthermore, if the weights in $\mathcal{A}$
are pairwise linearly independent, then the associated complex projective space $\mathbb{P}(\C^n_\mathcal{A})$
has exactly $n$ $T$-fixed points: the lines $\C_{\alpha_i}$. One also verifies that 
$$\E_{T}([\C_{\alpha_i}], \mathbb{P}(\C^n_\mathcal{A}))=\prod_{j\neq i}(\alpha_i-\alpha_j).$$ 
See \cite{ar:eu}, Remark 2.4.1-1.
\end{ex}

\begin{rem}
The inverse of the equivariant Euler class 
coincides with the equivariant multiplicity at a nondegenerate fixed point (\cite{bri:eqchow}, Section 4).  
\end{rem}

\begin{prop}[Atiyah-Bott Localization formula, \cite{ar:eu}]
Let $X$ be a complex projective variety. Suppose that a torus $T$
acts on $X$ with only a finite number of fixed points. Then
$$
\int_X \mu = \sum_{x\in X^T} \frac{\mu|_{x}}{\E_{T}(x,X)},
$$
for any $\mu \in H^*_T(X)$.
\hfill $\square$  
\end{prop}

Let $(X,x)$ be a rational cell. Then, by Proposition \ref{tang.cell}, 
$X$ admits a closed $T$-equivariant embedding into its tangent space $T_xX$. 
Notice that there are only a finite number of codimension-one  
subtori $T_1, \ldots, T_m$ of $T$ for which $X^{T_j}\neq X^T$, since 
each one of them is contained in the kernel of a weight of $T$ in $T_xX$.


\begin{thm}[\cite{ar:eu}, \cite{bri:ech}]
Let $(X,x)$ be a rational cell of dimension $n$. Let $\pi:X\to \C^n$ be the equivariant finite surjective map from Theorem \ref{arabiamap.thm}.
Then
\begin{enumerate}[(a)]
 \item The induced morphism in cohomology
$$\pi^*:H_c^{2n}(\C^n)\longrightarrow H_c^{2n}(X)$$
is an isomorphism and satisfies $\int_Y \pi^*(\mu)=deg(\pi) \int_{\C^n}\mu$,
where $deg(\pi)$ is the cardinality of a general fibre of $\pi$.
This formula also holds in equivariant cohomology, in particular
$$\E_{T}(0,\C^n)=deg(\pi)\cdot \E_{T}(x_0,X).$$

 \item $\E_T(X,x)=c\displaystyle \prod_{T_i}\E_{T}(X^{T_i},x)$, where $c$ is a positive rational number, and 
       the product runs over the finite number of codimension-one subtori $T_i$ of $T$
       for which $X^{T_i}\neq X^T$.
\end{enumerate} 
\end{thm}

\begin{proof}
We refer the reader to \cite{ar:eu}, Proposition 3.2.1-1, for the proof of part (a).
Finally, part (b) follows from Remark 5.3 and \cite{bri:ech}, Theorem 18 (iii).
\end{proof}

Let $(X,x)$ be a rational cell. At the beginning of this section
it was shown  
that  
$\E_T(x,X)$ splits into a product of characters.
The following result provides a geometric interpretation of this factorization.

\begin{cor}\label{characters.cell.cor}
Let $(X,x)$ be a rational cell of dimension $n$. 
Suppose that $X$ contains only a finite number of closed irreducible $T$-invariant curves 
$\mathcal{C}_i$, $i=1,\ldots,n$. 
Let $\chi_i$ be the character associated with the action of $T$ on $\mathcal{C}_i$.
Then 
$$\E_T(x,X)=c\cdot \chi_1 \cdots \chi_n,$$
where $c$ is a positive rational number.  
\end{cor}

\begin{proof}
There is only a finite number of codimension-one subtori $T_i$
such that $X^{T_i}\neq X^T$. 
Notice that  
$T$ acts on each $X^{T_i}$
through its quotient $T/{T_i}\simeq \C^*$. 
Because $x$ is an attractive fixed point of $X$, 
we can assume, without loss of generality, 
that $x$
is an attractive fixed point of each $X^{T_i}$, for the induced action of $\C^*\simeq T/T_i$. 
It follows from Corollary 2 of Sect. 1.4 of \cite{bri:rat} that  
$X^{T_i}=\mathcal{C}_i$. 
Moreover, by Theorem 1.1 of \cite{bri:rat}, each $X^{T_i}$ is rationally smooth at $x$.
Hence, each $X^{T_i}$ is a one-dimensional rational cell with attractive fixed point $x$ (see Lemma \ref{1dim.ch}
for a characterization of these cells). The result can now be deduced from Theorem 5.5 and Example 5.2. 
\end{proof}

\section{Local Indices and Module generators for the equivariant cohomology of $\Q$-filtrable GKM varieties.}

We 
supply 
a method for building
canonical free module
generators on the equivariant cohomology
of any $\Q$-filtrable GKM-variety.
Our findings here extend the work of Guillemin-Kogan (\cite{gk:morse}),
on the equivariant $K$-theory of orbifolds,  
to the equivariant cohomology of a much larger class of singular varieties.

\smallskip

Let $X$ be a $\Q$-filtrable GKM-variety. In other words, $X$
is a normal projective $T$-variety with only a finite number of fixed points
and $T$-invariant curves. 
Moreover, 
there exists a
$BB$-decomposition of $X$
as a disjoint union of rational cells, say $(C_1,x_1), \ldots, (C_m,x_m)$, 
each one 
containing $x_i\in X^T$ as its unique attractive fixed point. 
This decomposition
induces a filtration of $X$,
$$
\emptyset=X_0\subset X_1\subset X_2 \ldots \subset X_m=X, 
$$
by closed invariant subvarieties $X_i$, 
so that
each difference $X_i\setminus X_{i-1}$ equals $C_i$, for $i=1,\ldots,m$.
The key observation here is provided by Theorem \ref{eqforfiltration.thm}. It states that every $X_i$
is equivariantly formal and is made up of rational cells. 
In consequence,
GKM-theory can be applied to each $X_i$. We will refer to $X_i$
as the {\em $i$-th filtered piece of $X$}, and $m$ will be called
the {\em length of the filtration}.

\medskip

Denote by $x_1,\ldots,x_m$ the fixed points of $X$.
The filtration induces a total ordering of the fixed points, namely,
$$x_1<x_2<\ldots<x_m.$$

Let $(C_i,x_i)$ be a rational cell of $X$. From the previous section,
we know that $$H^*_{T,c}(C_i)\simeq H^*_{T}(C_i,C_i-\{x_i\}) \simeq H^*_T(x_i),$$ 
where 
the second isomorphism 
is provided by the Thom class $\mathcal{T}_i$, a well-known element of $H^*_{T}(C_i,C_i-\{x_i\})$. 
When restricted
to $H^*_T(x_i)$, the Thom class $\mathcal{T}_i$ becomes a product of linear polynomials: 
the Euler class $\E_T(c_i,C_i)$.

\medskip

In section 4 we built non-equivariant short exact sequences of the form
$$\xymatrix{0 \ar[r]& H^{2k}_c(C_i) \ar[r]& H^{2k}(X_i) \ar[r]& H^{2k}(X_{i-1})\ar[r] & 0},$$
for every $i$. Since the spaces involved have zero cohomology in odd degrees, then these
short exact sequences naturally generalize to the equivariant case, 
so we also have equivariant short exact sequences
$$\xymatrix{0 \ar[r]& H^{2k}_{T,c}(C_i) \ar[r]& H^{2k}_T(X_i) \ar[r]& H^{2k}_T(X_{i-1})\ar[r] & 0},$$
for each $i$. On the other hand, by equivariant formality, the singular equivariant cohomology of each $X_i$
injects into $H^*_T(X_i^T)=\oplus_{j\leq i}H^*_T(x_j)$. 

In summary, for each $i$,
we have the commutative diagram
$$
\xymatrix{
0 \ar[r]& H^{*}_{T,c}(C_{i+1}) \ar[r]\ar[d]& H^{*}_T(X_{i+1})          \ar[r]\ar[d]      & H^{*}_T(X_i)\ar[r]\ar[d]        & 0 \\
0 \ar[r]& H^{*}_T(x_{i+1})     \ar[r]       & \oplus_{j\leq i+1} H^{*}_T(x_j) \ar[r]& \oplus_{j\leq i} H^*_T(x_j) \ar[r] & 0 \\
}
$$
\label{lift.diagram}
where the vertical maps are all injective. Indeed, 
such 
maps correspond to 
the various restrictions to 
fixed point sets. 
We will use this diagram
to 
build cohomology generators. The next two lemmas
are inspired in Theorem 2.3 and Proposition 4.1 of \cite{hhh:c}, 
where Kac-Moody flag varieties are studied.

\begin{lem}
Let $X$ be a $\Q$-filtrable variety. Then there exists a non-canonical
isomorphism of $H^*_T$-modules
$$
H^*_T(X) \simeq \bigoplus_{x_i\in X^T}\E_T(C_i,x_i)H^*_T(pt),
$$
which is compatible with restriction to the various $i$-th filtered pieces $X_i\subset X$.
\end{lem}

\begin{proof}
We argue by induction on the length of the filtration. The case $m=1$ is simple,
because it corresponds to $X=\{x_1\}$, a singleton.
Assuming that we have proved the assertion for $m$, 
let us 
prove the case $m+1$. 
Substitute $i=m$ in the commutative diagram above. 
Then $$H^*_T(X_{m+1})=H^*_T(X)\simeq H^*_{T,c}(C_{m+1})\oplus H^*_T(X_m).$$
By induction, $H^*_T(X_m)\simeq \prod_{i\leq m}\E_T(C_i,x_i)H^*_T(pt)$.
So the claim for $m+1$ follows directly from the isomorphism 
$H^*_{T,c}(C_{m+1})\simeq \E_T(C_{m+1},x_{m+1})H^*_T(pt).$
\end{proof}

The isomorphism of the previous lemma is not canonical because the
cellular decomposition of $X$ depends on a particular choice 
of generic one-parameter subgroup and a compatible ordering
of the fixed points. 

\medskip

Given a class $\mu \in H^*_T(X)$, denote by $\mu(x_i)$ its restriction to the fixed point $x_i$. 

\begin{lem}
Let $X$ be a projective $T$-variety. 
Assume that $X$ is $\Q$-filtrable and 
let $x_1<x_2<\ldots<x_m$ be the order relation on $X^T$ 
compatible with the filtration of $X$.
For each $i$, let $\varphi_i\in H^*_T(X)$ be a class such that
$$\varphi_i(x_j)=0 \; {\rm \it for} \; j < i,$$
and 
$$\varphi_i(x_i) \; {\rm \it is \, a \, scalar \, multiple \, of} \; \E_T(i,C_i).$$
Then the classes $\{\varphi_i\}$ generate $H^*_T(X)$ freely as a module over $H^*_T(pt)$. 
\end{lem}

\begin{proof}
Since $X$ is equivariantly formal, we know that $H^*_T(X)$ injects into $H^*_T(X^T)$
and is a free $H^*_T$-module of rank $m=|X^T|$. First, we show that the $\varphi_i$'s 
are linearly independent. 
Arguing by contradiction, suppose there is a
linear combination 
$$\sum_{i=0}^m f_i\varphi_i=0,$$
with $f_i \in H^*_T$, not all of them zero.
Let $k$ be the minimum of the set $\{i\,|\, f_i\neq 0\}$. Then 
we have
$$f_k\varphi_k + f_{k+1}\varphi_{k+1}+\ldots f_m\varphi_m=0$$
where $f_k\neq 0$.
Let us restrict this linear combination to $x_k$. Then
$$f_k\varphi_k(x_k) + f_{k+1}\varphi_{k+1}(x_k)+\ldots f_m\varphi_m(x_k)=0.$$
But $\varphi_{\ell}(x_k)=0$ for all $\ell>k$. Thus we obtain
$$f_k\varphi(x_k)=0.$$ 
However, $\varphi(x_k)$ is a non-zero multiple of the Euler class $\E_T(x_k,C_k)$ and, as such,
it is non-zero. We conclude that $f_k$ must be zero. This is a contradiction.

To conclude the proof, we need to show that the $\varphi_i$'s generate $H^*_T(X)$
as a module. 
But this is a routine exercise, using induction on the length of the filtration of $X$
(the base case being trivial). The commutative
diagram of page \pageref{lift.diagram} then disposes of the inductive step. 
\end{proof}

As for the existence of classes satisfying Lemma 6.2, 
we will show that they can always be constructed on GKM-varieties. 
First, we need two technical lemmas.

\begin{lem}
Let $X$ be a normal projective $T$-variety 
with finitely many fixed points. 
Choose a generic one-parameter subgroup 
and write $X$ as
$X=C\sqcup Y$, 
where $$C=\{z\in X \, | \, \lim_{t\to 0}tz=x\}$$ 
is the stable cell of $x\in X^T$, 
and 
$Y$ is closed and $T$-stable.
Then any closed irreducible $T$-stable curve that passes through $x$ is contained in
the Zariski closure of $C$.
\end{lem}

\begin{proof}
Let $\ell$ be a closed irreducible $T$-stable curve passing through $x$. 
Recall that $\ell$  
is the closure of a one-dimensional orbit $Tz$. 
Moreover, $\ell=\overline{Tz}$
has two fixed points, namely, $x$ and a fixed point $y_{i(\ell)}$ contained {\em necessarily} in $Y$. 
We claim that $z \in C$. 
For otherwise, $\displaystyle \lim_{t\to 0}tz=y_{i(\ell)}$, 
which implies that $z$ belongs to the stable subvariety of $y_{i(\ell)}$.
Since $Y$ is $T$-invariant and closed, then $\ell=\overline{T z}\subset Y$. 
That is, $x\in \partial{\ell}$ would belong to $Y$, which is absurd. 
Thus $z\in C$.

The fact that $C$ is also $T$-stable gives the inclusion $Tz\subset C$.
We conclude that $\ell=\overline{Tz}\subset \overline{C}$.     
\end{proof}

\begin{lem}
Let $X$ be a normal projective variety on which a torus acts with a finite number of fixed points
and one-dimensional orbits. Suppose $X$ is equivariantly formal and  
there is a generic one-parameter subgroup 
such that $X$ can be written as a disjoint union
$X=C\sqcup Y$, 
where $$C=\{z\in X \, | \, \lim_{t\to 0}tz=x\}$$ 
is a rational cell with unique attractive fixed point $x\in X^T$, 
and 
$Y$ is closed and $T$-stable.
Then the
cohomology class 
$\tau\in \oplus_{w\in X^T} H^*_T(w)$, 
defined by
$$\tau(x)=\E_T(x,C) {\rm \; \; and \; \;} \tau(y)=0 {\rm \; for \; all \;} y\in Y^T,$$
belongs to the image of $H^*_T(X)$ in $H^*_T(X^T)$.
\end{lem}

\begin{proof}
The hypotheses imply that $X$ is a GKM-variety. As a result, 
the equivariant cohomology of $X$ can be described 
by the GKM-relations
of Theorem \ref{gkm.thm}. So, 
to prove the lemma, it is enough to verify that $\tau$
satisfies such relations.

Because $\tau$ restricts to zero at every fixed point except $x$, we need only show that
$$\tau(x)=\tau(x)-\tau(y)=\E_T(x,C)$$ is divisible by $\chi_{i}$ whenever the fixed points $x\in C$ and $y_i\in Y^T$
are joined by a $T$-curve $\ell_i$ in $X$, and $T$ acts on $\ell_i$ through $\chi_i$. Let $p$ be the total
number of $\ell_i$'s.

By Lemma 6.3, the curve $\ell_i$ is contained in the Zariski closure $\overline{C}$
of $C$. In fact, $\ell_i\setminus \{y_i\}\subset C$. 
Also, it follows from Theorem \ref{cell.curves.thm} that $p=dim(C)$.
Thus,  
using Corollary \ref{characters.cell.cor}, we conclude that $\E_T(x,C)$ is a non-zero multiple of the $\chi_i$'s. 
In short, $\tau$ belongs to $H^*_T(X)$.
\end{proof}

\medskip

It is noticeable that, in the previous lemmas, no assumption on the irreducibility of $X$ has been made. 
Surely we allow for some flexibility in this matter, since the various filtered pieces $X_i$
of a $\Q$-filtrable space $X$ need not be irreducible.

\medskip

\begin{thm}
Let $X$ be a $\Q$-filtrable GKM-variety.
Then cohomology generators $\{\varphi_i\}$ of $H^*_T(X)$ with the properties described
in Lemma 6.2 exist.  
\end{thm}

\begin{proof}
We proceed by induction on $m$, the length of the filtration of $X$. 
If $m=1$, then $X=\{x_1\}$ and the statement is clear, since we can just choose $\varphi_1=1$.
Assuming we have proved the statement for varieties with a filtration of length $m$, let us prove the
case when the length is $m+1$. First, notice that 
$X_{m+1}=X$ and, 
by the inductive hypothesis, 
there are classes $\varphi_1,\ldots,\varphi_m\in H^*_T(X_m)$
which satisfy the desired properties in $H^*_T(X_{m})$. 
Using the commutative diagram of page \pageref{lift.diagram},
we can lift them to classes $\tilde{\varphi_1},\ldots,\tilde{\varphi_m}$ 
which still satisfy the required conditions, though this time they lie in $H^*_T(X_{m+1})=H^*_T(X)$.
In consequence, we just need to construct a class $\varphi_{m+1}\in H^*_T(X)$ with
the sought-after qualities. 
So set $\varphi_{m+1}(x_{m+1})=\E_T(x_{m+1},C_{m+1})$ and $\varphi_{m+1}(x_j)=0$
for all $j\leq m$. 
By Lemma 6.4, this class surely belongs to $H^*_T(X)$.
Thus the result also holds for varieties with a filtration of length $m+1$. 
This provides the inductive step in view of Lemma 6.2. 
\end{proof}


\begin{dfn}
Let $X$ be a $\Q$-filtrable $T$-variety.
%
Fix an ordering of the fixed points, say
$x_1<x_2< \ldots <x_m$.
Given $\mu \in H^*_T(X)$, we define its {\bf local index at $x_i$}, denoted $I_i(\mu)$, 
by the following formula:
$$I_i(\mu)=\int_{X_i}p_i^*(\mu),$$
where $p_i:X_i\to X$ denotes the inclusion of the $i$-th filtered piece into $X$.
It follows from the definition that assigning local indices yields $H^*_T$-linear morphisms
$$I_i:H^*_T(X)\to H^*_T(pt).$$
\end{dfn}

Using the localization formula (Proposition 5.4), one can easily prove the following

\begin{lem}
The local index of $\mu$ at $x_i$ satisfies
$$I_i(\mu)=\sum_{j\leq i}\frac{\mu(x_j)}{\E_T(x_j,X_i)},$$
where $\mu(x_j)$ denotes the restriction of $\mu$ to $x_j$. \hfill $\square$ 
\end{lem}

\begin{cor}
Let $x_i\in X^T$, be a fixed point.
Suppose that $\mu\in H^*_T(X)$ is a cohomology class that satisfies
$\mu(x_j)=0$ for all $j<i$. Then $$\mu(x_i)=I_i(\mu)\E_T(x_i,X_i).$$ \hfill $\square$  
\end{cor}


Our most important result in this Section is the
following generalization of the work of 
Guillemin and Kogan (\cite{gk:morse}, Theorems 1.1 and 1.6)
to $\Q$-filtrable GKM-varieties.

\begin{thm}\label{eulergenerators.thm}
Let $X$ be a $\Q$-filtrable GKM-variety.
Let $x_1<x_2<\ldots<x_m$ be the order relation on $X^T$ 
compatible with the filtration of $X$.
Then, for each $i=1,\ldots,m$, 
there exists a unique class $\theta_i\in H^*_T(X)$ with the following properties:
\begin{enumerate}[(i)]
\item $I_i(\theta_i)=1$,

\item $I_j(\theta_i)=0$ for all $j\neq i$,

\item the restriction of $\theta_i$ to $x_j \in X^T$ is zero for all $j<i$, and

\item $\theta_i(x_i)=\E_T(i,C_i)$.
\end{enumerate}
Moreover, the $\theta_i$'s generate $H^*_T(X)$ freely as a module over $H^*_T(pt)$. 
\end{thm}

\begin{proof}
By Theorem 6.5,
choose a set of free generators $\{\varphi_i\}$ which satisfy the properties
described in Lemma 6.2, 
together with the additional condition 
$\varphi_i(x_i)=\E_T(i,C_i)$.

\smallskip

Given $i$, 
notice that $I_j(\varphi_i)=0$, for all $j<i$, and $I_i(\varphi_i)=1$.
We will show that we can modify these $\varphi_i$'s accordingly
to obtain the generators $\theta_i$.
In fact, given $i\in \{1,\ldots,m\}$, 
the only obstruction to setting 
$\theta_i=\varphi_i$ is that
$I_j(\varphi_i)$ can be non-zero for some $j>i$.

\smallskip

Let $i\in \{1,\ldots,m\}$. 
If $I_j(\varphi_i)=0$ for all $j>i$, 
then let $\theta_i=\varphi_i$.
Otherwise, proceed as follows.
Let $k_0$ be the minimum of all $k>i$ such that $I_k(\varphi_i)\neq 0$.
Define $\Psi_i=\varphi_i-I_{k_0}(\varphi_i)\varphi_{k_0}$. 
Let us compute the local indices of $\Psi_i$.
Clearly, if $j<i$, we have $I_j(\Psi_i)=0$.
Also, if $j=i$, then $I_i(\Psi_i)=1$. 
It is worth noticing that $\Psi_i$ restricts to 0 at each $x_j$ with $j<i$.
Now if $j$ satisfies
$i<j\leq k_0$, then $I_j(\Psi_i)=0$. 
So, arguing by induction, 
we can provide a class $\widetilde{\Psi_i}$ such that $I_j(\widetilde{\Psi_i})=0$ for all $j\neq i$, 
and $I_i(\widetilde{\Psi_i})=1$.
Thus, set $\theta_i=\widetilde{\Psi_i}$.
Working on each $i$ at a time, we conclude that there exist classes $\theta_i$ satisfying
conditions (i)-(iv) of the Theorem.

\smallskip

Let us now prove uniqueness. Suppose there are classes $\{\theta_i\}$ and $\{\theta_i'\}$ satisfying 
all the properties of the theorem. Fix $i$ and let $\tau=\theta_i-\theta_i'$.
It is clear that $\tau$ is an element of $H^*_T(X)$ whose local index $I_j(\tau)$
is zero for all $j$. Suppose that $\tau$ is not zero. Then, since $H^*_T(X)$
injects into $H^*_T(X^T)$, there should be a $k$ such that $\tau(x_k)\neq 0$.
Take the minimum of all $k$'s for which $\tau(x_k)\neq 0$. Denote this minimum by $s$.
Then, by Corollary 6.8, one would have $\tau(x_s)=I_s(\tau)\E_T(x_s,X_s)=0$. 
This is absurd. Therefore $\tau=0$.
Since $i$ can be chosen arbitrarily, we conclude that $\theta_i=\theta_i'$ for all $i$.

\smallskip

Finally, notice that properties (iii) and (iv) together with Lemma 6.2 imply that
the $\theta_i$'s freely generate $H^*_T(X)$. We are done.
\end{proof}

\section{Rational Cells and Standard Group Embeddings}

Previously, we have developed the
theory of $\Q$-filtrable varieties. 
In this last section,
we provide the theory with a large class of examples, namely,
rationally smooth standard embeddings. 
We show that 
these varieties
admit $BB$-decompositions into rational cells (Theorem \ref{ratsmisqfiltrable.cor}).
Thus, they are $\Q$-filtrable and satisfy Theorems \ref{eqforfiltration.thm}
and \ref{eulergenerators.thm}.

\smallskip

First, let us set the stage. 
An affine algebraic monoid $M$ is called {\bf reductive} if it is irreducible, normal, and its
unit group is a reductive algebraic group. See \cite{re:lam} for many of the details.
A reductive monoid is called
{\bf semisimple} if it has a zero element, and its unit group 
has a one-dimensional center. 

Let $M$ be a reductive monoid with zero. 
Denote by $G$ its unit group and by $T$ a maximal torus of $G$.
Associated to $M$, there is a torus embedding $\overline{T}\subset M$
defined as follows $$\overline{T}=\{x \in M \;|\; xt=tx, \, {\rm \, for \, all \; } t\in T\}.$$ 
Certainly, $T\subseteq \overline{T}$.
Let $E(\overline{T})$ be the idempotent set of $\overline{T}$; that is, 
$$E(\overline{T})=\{e\in \overline{T}\,|\, e^2=e\}.$$
%



The {\bf Renner monoid}, $\mathcal{R}$, is defined to be $\mathcal{R}:=\overline{N_G(T)}/T$.
It is a finite monoid
whose group of units is $W$ (the Weyl group) 
and contains $E(\overline{T})$
as idempotent set. In fact, 
any $x\in \mathcal{R}$ can be written as $x=fu$, where $f\in E(\overline{T})$ and $u\in W$.
Recall that $W$ is generated by reflections $\{s_\alpha\}_{\alpha \in \Phi}$, where $\Phi$ is the set of roots
of $G$ with respect to $T$.

\smallskip

Denote by $\mathcal{R}_k$ the set of elements of rank $k$ in $\mathcal{R}$, 
that is, 
$$\mathcal{R}_k=\{x\in \mathcal{R} \, | \, \dim{Tx}=k \,\}.$$

\begin{dfn}\label{standard.dfn}
Let $M$ be a reductive monoid with unit group $G$ and zero element $0\in M$.
There exists a central one-parameter subgroup 
$\epsilon:\C^*\to G$ with image $Z$ contained in the center of $G$, 
that 
converges to $0$ (\cite{bri:mon}, Lemma 1.1.1). 
Then $\C^*$ acts attractively on $M$ via $\epsilon$,
and hence the quotient
\[
\mathbb{P}_\epsilon(M) = [M\setminus\{0\}]/{\C^*}
\] 
is a normal projective variety.
Notice also that $G\times G$ acts on $\mathbb{P}_\epsilon(M)$ via
\[
G\times G\times \mathbb{P}_\epsilon(M)\to \mathbb{P}_\epsilon(M),\;
(g,h,[x])\mapsto [gxh^{-1}].
\]
Furthermore, $\mathbb{P}_\epsilon(M)$ is a normal projective embedding 
of the reductive group $G/Z$.
In the sequel, $X=\P_\epsilon(M)$ will be called a {\bf Standard Group Embedding}.
\end{dfn}
 
For an up-to-date description of these and other embeddings, see \cite{ab:em}.

\begin{ex}
Let $G_0$ be a semisimple algebraic group over the complex numbers and let $\rho:G_0\to {\rm End}(V)$ be a representation of $G_0$. Define $Y_\rho$
to be the Zariski closure of $G=[\rho(G_0)]$ in $\mathbb{P}({\rm End}(V))$, the projective space associated with ${\rm End}(V)$.
Finally, let $X_\rho$ be the normalization of $Y_\rho$. 
By definition, $X_\rho$ is an standard group embedding of $G$. 
Notice that $M_\rho$, the Zariski closure of $\C^*\rho(G_0)$ in ${\rm End}(V)$, is a semisimple monoid whose group of units is $\C^*\rho(G_0)$.
Rationally smooth standard embeddings of the form $X_\rho$, with $\rho$ irreducible, 
have been classified combinatorially in \cite{re:ratsm}.
\end{ex}

We now come to the main result of this section. It states
that rationally smooth standard embeddings are equivariantly
formal for the induced $T\times T$-action.

\begin{thm}\label{ratsmisqfiltrable.cor}
Let $X=\mathbb{P}_\epsilon(M)$ be a standard group embedding. 
If $X$ is rationally smooth, then $X$ is $\Q$-filtrable.  
\end{thm}

\begin{proof}
Renner has shown that $X$ comes equipped with 
the following BB-decomposition:
$$X=\bigsqcup_{r\in \mathcal{R}_1}C_r,$$
where $\mathcal{R}_1=X^{T\times T}$. See 
Theorem 3.4 of \cite{re:hpoly}, and  
and Theorem 4.3 of \cite{re:hpolyirr} for more details. 
Our strategy is to show that
if $X$ is rationally smooth, then each cell $C_r$ is rationally smooth.

With this purpose in mind, 
we call the reader's attention to the fact that, 
in the terminology of \cite{re:hpoly}, 
$M$ is quasismooth (Definition 2.2 of \cite{re:hpoly}) if and only if
$M\setminus\{0\}$ is rationally smooth.
The equivalence between these two notions 
follows from Theorem 2.1 of \cite{re:hpoly} and Theorems
2.1, 2.3, 2.4 and 2.5 of \cite{re:ratsm}. 

Next, by Lemma 4.6 and Theorem 4.7 of \cite{re:hpoly}, each $C_r$ equals $$U_1\times C_r^* \times U_2,$$
where the $U_i$'s are affine spaces. 
Moreover, if we write $r\in \mathcal{R}_1$ as $r=ew$, with $e\in E_1(\overline{T})$ and $w\in W$, 
then $C_r^*=C_e^*w$. 
So it is enough to show that $C_e^*$ is rationally smooth, for $e\in E_1(\overline{T})$.

By Theorem 5.1 of \cite{re:hpoly}, it follows that, if $X=\mathbb{P}_\epsilon(M)$ is rationally smooth, then
$$C_e^*=[f_eM(e)]/Z,$$
for some unique $f_e \in E(\overline{T})$, where $M(e)=M_eZ$ and $M_e$ is rationally smooth (Theorem 2.5 of \cite{re:ratsm}).
Furthermore, the proof of Theorem 5.1 of \cite{re:hpoly} also 
implies that $[e]$ is the 
zero element of the rationally smooth, reductive, affine monoid $M(e)/Z$. 
Additionally,
$$C_e^*=\{x\in M(e)/Z\,|\, \lim_{s\to 0}sx=[e]\},$$ 
for some generic one-parameter subgroup. 
Using Lemma 7.4 below, one concludes that $C_e^*$ is rationally smooth. 

Finally, since $X$ is normal, projective and admits a $BB$-decomposition
into rational cells, we have compiled all the necessary data to conclude that 
$X$ is $\Q$-filtrable. 
\end{proof}

\begin{lem} 
Let $M$ be a reductive monoid with zero. Suppose 
that zero $0$ is a rationally smooth point of $M$. Let
$f\in E(M)$, be an idempotent of $M$. Then
$0\in fM$ is a rationally smooth point of the
closed subvariety $fM$.
\end{lem}

\begin{proof}
By Lemma 1.1.1 of \cite{bri:mon}, one can find
a one-parameter subgroup $\lambda:\C^*\to T$, with image $S$,
such that $\lambda(0)=f$. Notice that
$$fM=\{x\in M\,|\,\lambda(t)x=x \, , \forall \, t\in \C^*\}.$$
That is, $fM$ is the fixed point set of the subtorus $S$ of $T$.
Thus, by Theorem 1.1 of \cite{bri:rat}, one concludes that $0$
is also a rationally smooth point of $fM$.
\end{proof}

We conclude this section providing a partial converse to Theorem \ref{ratsmisqfiltrable.cor}.

\begin{thm}
Let $X=\P_\epsilon(M)$ be a standard embedding. 
Suppose that $X$ contains a unique closed $G\times G$-orbit.
If $X$ is $\Q$-filtrable, then 
$X$ is rationally smooth. 
\end{thm}

\begin{proof}
Since
$X$ contains a unique closed $G\times G$-orbit,  
it follows from \cite{re:lam}, Chapter 7, that
$W\times W$ acts transitively on $\mathcal{R}_1$, the set of 
representatives of the $T\times T$-fixed points of $X$.
Because $X$ is irreducible, there
exists a unique cell, say $C_\sigma$, with $\sigma \in \mathcal{R}_1$, 
such that $X=\overline{C_{\sigma}}$. 
By assumption,
$C_{\sigma}$ is rationally smooth at $\sigma$, and so, $X$ is rationally smooth at $\sigma$.
We claim that $X$ is rationally smooth at every $r\in \mathcal{R}_1=X^{T\times T}$. 
Indeed, by the previous remarks, $r=w\cdot \sigma \cdot v$, for
some $(w,v)\in W\times W$ and rational smoothness is a local property 
invariant under homeomorphisms. Now Lemma 7.6 below concludes the proof.
\end{proof}

\begin{lem}
Let $X$ be a projective $T$-variety with a finite number of fixed points $x_1,\ldots, x_m$. 
Then $X$ is rationally smooth at every $x\in X$ if and only if $X$ is rationally smooth at every fixed point $x_i$.
\end{lem}

\begin{proof}
One direction is clear. For the converse, 
pick a generic one-parameter subgroup $\lambda:\C^*\to T$
such that $X^T=X^{\C^*}$. 
Let $x\in X$. 
Then, there exists $x_k\in X^T$ such that $x_k=\displaystyle \lim_{t\to 0}tx$ (BB-decomposition).
Moreover, since $X$ is rationally smooth at $x_k$, there exists a
neighborhood $V_k$ of $x_k$ with the property that $X$ is rationally smooth at every $y\in V_k$.
By construction, there exists $s\in \C^*$ satisfying $sx\in V_k$. 
To see this, simply notice that we can find a sequence
$\{s_n\}\subset \C^*$ for which $s_n\cdot x$ converges to $x_k$, i.e. there is $N$ such that 
$s_n\cdot x$ belongs to $V_k$, for all $n\geq N$. Now setting $s=s_N$ yields $s\cdot x\in V_k$. 
In other words, $sx$ is a rationally smooth point of $X$. 
But the set of rationally smooth points is $T$-invariant.
Hence, $x$ is a rationally smooth point of $X$. 
Inasmuch as the point $x$ was chosen arbitrarily, the argument is complete.
\end{proof}

In the author's thesis it was shown 
that all standard embeddings are $T\times T$-skeletal. 
Consequently, rationally smooth standard embeddings 
are also GKM-varieties.
In a forthcoming paper (\cite{go:standard})
we find explicitly all the GKM-data 
(i.e. fixed points, invariant curves and associated characters) 
of any rationally smooth standard embedding $\P_\epsilon(M)$,  
and describe $H^*_{T\times T}(\P_\epsilon(M))$ as a complete combinatorial invariant of $M$.  
The results will appear elsewhere.

\end{document}